\documentclass[11pt]{amsart}

\usepackage{amsmath,amsfonts,amsthm,amssymb,color}
\usepackage[latin1]{inputenc}
\usepackage[all,cmtip]{xy}
\usepackage{fullpage}

\numberwithin{equation}{section}

\newtheorem{theorem}[equation]{Theorem}
\newtheorem{prop}[equation]{Proposition}
\newtheorem{lemma}[equation]{Lemma}
\newtheorem{cor}[equation]{Corollary}

\newtheorem{defn}[equation]{Definition}
\theoremstyle{remark}
\newtheorem{remark}[equation]{Remark}

\newcommand{\cN}{{\mathcal{N}}}
\newcommand{\cP}{{\mathcal{P}}}

\newcommand{\cO}{{\mathcal{O}}}

\newcommand{\N}{\mathbb{N}}
\newcommand{\bbH}{\mathbb{H}}
\newcommand{\bbS}{\mathbb{S}} 
\newcommand{\R}{\mathbb{R}}
\newcommand{\C}{\mathbb{C}}

\newcommand{\bbQ}{\mathbb{Q}}

\newcommand{\cL}{\mathcal{L}}

\newcommand{\reals}{{\bf R}}

\newcommand{\pa}{\partial}

\newcommand{\li}{{\rm li}}

\newcommand{\calO}{\mathcal{O}}

\newcommand{\gL}{\mathcal{L}}
\newcommand{\pgL}{\mathcal{L}_p} 

\def\abs#1{\lvert #1\rvert}

\newcommand{\norm}[1]{|\!|#1|\!|}
\def\BIgl({\mathopen{\hbox{\smaller${\biggl(}$}}}
\def\BIgr){\mathclose{\hbox{\smaller${\biggr)}$}}}
\def\BIglv|{\mathopen{\hbox{\smaller${\biggl|}$}}}
\def\BIgrv|{\mathclose{\hbox{\smaller${\biggr|}$}}}
\def\bil({\mathopen{\raise.25pt\hbox{\smaller${\bigl(}$}}}
\def\bir){\mathclose{\raise.25pt\hbox{\smaller${\bigr)}$}}}
\def\bilv|{\mathopen{\raise.25pt\hbox{\smaller${\bigl|}$}}}
\def\birv|{\mathclose{\raise.25pt\hbox{\smaller${\bigr|}$}}}

\renewcommand{\Re}{\mathop{\rm Re}}

\def\udstrutpt#1#2{\noindent\vbox to #1pt{}\noindent\lower#2pt\vbox{}\ignorespaces}

\def\zeroint_#1{\mathchoice
{\sideset{ ^{\hskip5pt 0\hskip-5pt}}{}\int\nolimits_{\hskip-5pt #1}}
{^{\hskip4pt 0\hskip-4pt}\int_{#1}} 
{\sideset{ ^{\hskip5pt 0\hskip-5pt}}{}\int}
{\sideset{ ^{\hskip5pt 0\hskip-5pt}}{}\int}
}

\def\smint{\raise.5pt\hbox{\smaller$\dsty \int$}}
{\makeatletter
\gdef\citemt[#1]#2{[\hyper@@link[cite]{}{cite.#2}{#1}]}
\gdef\citelcs[#1]#2{(loc.\ cit.)}
\gdef\citelc[#1]#2{(\hyper@@link[cite]{}{cite.#2}{loc.\ cit.})}
\gdef\citeib[#1]#2{[\hyper@@link[cite]{}{cite.#2}{ibidem}, #1]}
\gdef\citeibna#1{(\hyper@@link[cite]{}{cite.#1}{ibidem})}
}

\def\udstrutpt#1#2{\noindent\vbox to #1pt{}\noindent\lower#2pt\vbox{}\ignorespaces}

\renewcommand{\qed}{\hspace*{\fill} \setlength{\unitlength}{1mm}
\begin{picture}(2.5,2.5)
  \put(0,0){\framebox(2.5,2.5){}}
 \end{picture}
\setlength{\unitlength}{1pt}}

  \definecolor{pink}{rgb}{1,0,1}

\begin{document}
\title{Dynamics and zeta functions on conformally compact manifolds}

\author[Julie Rowlett]{Julie Rowlett}
\address{Max Planck Institut f\"ur Mathematik \\ Vivatgasse 7 \\ D-53111 Bonn} 
\email{rowlett@mpim-bonn.mpg.de}

\author[Pablo Su\'arez-Serrato]{Pablo Su\'arez-Serrato}
\address{  Current:  Departament de Matem\`atica Aplicada 1 \\ 
Universitat Polit\`ecnica de Catalunya\\ 
Barcelona \\ Permantent:   Instituto de Matem\'aticas,
              Universidad Nacional Auton\'oma de M\'exico\\
              Ciudad Universitaria, Coyoac\'an,
              04510. M\'exico, D. F.  } \email{p.suarez-serrato@matem.unam.mx}

\author[Samuel Tapie]{Samuel Tapie}
\address{Laboratoire Jean Leray\\ Universit de Nantes\\
2, rue de la Houssini\`ere - BP 92208\\ F-44322 Nantes Cedex 3, France} 
\email{samuel.tapie@univ-nantes.fr}


\begin{abstract}
In this note, we study the dynamics and associated zeta functions of conformally compact manifolds with variable negative sectional curvatures. We begin with a discussion of a larger class of manifolds known as \em convex co-compact manifolds with variable negative curvature. \em Applying results from dynamics on these spaces, we obtain optimal meromorphic extensions of weighted dynamical zeta functions and asymptotic counting estimates for the number of weighted closed geodesics. A 
meromorphic extension of the standard dynamical zeta function and the \emph{prime orbit theorem} follow as corollaries. Finally, we investigate interactions between the dynamics and spectral theory of these spaces. 
\end{abstract}
\maketitle

\section{Introduction}
\emph{Conformally compact} manifolds are a class of non-compact manifolds with variable curvature, introduced by C. Fefferman and C. R. Graham \cite{fg} to study conformal invariants. A conformally compact manifold is an open manifold with compact boundary whose Riemannian metric is conformally compact. These metrics generalize the Poincar\'e model of hyperbolic space, so conformally compact metrics are also known as Poincar\'e metrics. 
This paper focuses on the dynamics of those conformally compact manifolds with variable curvature, in the purpose of further applications to their spectral theory. 

In 2001, Perry \cite{plength} demonstrated a \em prime orbit theorem \em for the geodesic flow on convex co-compact hyperbolic manifolds which predicts the asymptotic behavior of the number of prime orbits of the associated flow in the spirit of the prime number theorem. Perry's proof relies on a detailed study of the associated Selberg zeta function. However, another proof of this result was already known using purely dynamical methods. Indeed, the prime orbit theorem of \cite{plength} follows almost immediately from Parry and Pollicott's prime orbit theorem for Axiom A flows restricted to a basic set \cite{pp83}. One goal of this work is to increase communication between dynamics and geometric analysis. 

Convex co-compact hyperbolic manifolds are complete hyperbolic manifolds whose closed geodesics are contained in a compact set. One may also look at \em convex co-compact with variable (negative) curvature, i.e. complete manifold with negative sectional curvatures whose closed geodesics are contained in a compact set. \em In \S 2, we present in detail the notions of conformally compact and convex co-compact. We provide a complete proof of the following result because, even though it may well be known by experts, we could not locate a published proof. 

\begin{theorem}\label{th:ConfCvxCpct}
Let $(M,g)$ be a conformally compact manifold with negative sectional curvatures. Then $(M,g)$ is convex co-compact. The converse is true if $M$ is a surface or if $g$ has constant sectional curvatures.
\end{theorem}

Basic notions from dynamics are presented in \S 3. 

Connections between the dynamics and spectral theory on hyperbolic manifolds originate in Selberg's trace formula for compact locally symmetric spaces \cite{sel}. It has been a subject of major research since then; the work of Sullivan and of Lax and Phillips have been of particular influence. The Selberg trace formula is based on the use of a \emph{dynamical Zeta function}; these functions have turned out to be a powerful tool. In the late 1980s, Guillop\'e and Zworski generalized Selberg's trace formula to infinite area Riemann surfaces with finite geometry \cite{gz}, using a \emph{weighted dynamical Zeta function,} related to Selberg's zeta function. This was further extended to higher dimensions by Guillarmou and Naud in \cite{gn}.

On compact manifolds with \emph{variable} negative curvature, Duistermaat and Guillemin established an asymptotic trace formula using the wave kernel \cite{dg}, which also gives a relation between the spectrum and dynamics, even though it is harder to control. 
For manifolds with infinite volume and \emph{variable curvature}, one does not expect a Selberg trace formula. Nevertheless, an asymptotic version of Selberg's trace formula was obtained for compact perturbations of convex co-compact hyperbolic manifolds in \cite{row1}.\footnote{We have discovered a subtle gap in the proof of this result which, however does not affect the application by Borthwick and Perry \cite{bp} which shows that on compact perturbations of hyperbolic manifolds, the resonances determine the length spectrum.} Very few results are known for more general non-compact manifolds with variable curvature, such as the conformally compact or convex co-compact manifolds. As a first step, we study dynamical weighted zeta functions on such manifolds.

In \S 4, we use dynamical studies of Axiom A flows to demonstrate the following properties of a large class of dynamical zeta functions. Let $\pgL$ be the set of primitive closed geodesics, and $\gL$ be the set of all closed geodesics. For $\gamma \in \cL$ such that $\gamma = k \gamma_p$ with $\gamma_p \in \pgL$, let $l_p(\gamma) = l(\gamma_p)$.
\begin{theorem}\label{theo:WZeta}
Let $(M, g)$ be a conformally compact manifold whose sectional curvatures satisfy $-b^2\leq K_g\leq -a^2<0$, with non Abelian fundamental group and non-arithmetic length spectrum. Let $U : \gL\rightarrow \R$ be a weight \emph{which derives from the H\"older potential $W$}, and $\wp(W)$ be the pressure of $W$ (with respect to the geodesic flow). Then the weighted Zeta function
$$Z_U(s) = \exp \left( \sum_{\gamma \in \pgL} \sum_{k \in \N} \frac{1}{k}e^{-k s l_p(\gamma) + U(k\gamma)} \right)$$
converges absolutely on $\mathfrak{R}(s) > \wp(W)$. It admits a meromorphic extension to the half plane $\mathfrak{R}(s)> \wp(W)-\frac{\lambda\alpha}{2}$, where $\lambda\in[a,b]$ is the expansion factor of the geodesic flow on the non-wandering set, and $\alpha\in(0,1]$ is the H\"older exponent of $W$. This extension is in general optimal. Moreover, with the exception of a simple pole at $\wp(W)$, this extension is analytic and non-vanishing in an open neighborhood of $\{ \mathfrak{R}(s) \geq \wp(W)\}$. 
\end{theorem}

Based on the extension of the weighted Zeta functions, we demonstrate the following weighted prime orbit theorem. Let ${\cL}_{T}$ be the set of geodesics $\gamma$ of length at most $T$. 
 
\begin{theorem}[Weighted Prime Orbit Theorem]\label{theo:WPOT}
Let $(M, g)$ be a conformally compact manifold with negative sectional curvatures, with non Abelian fundamental group and non-arithmetic length spectrum. Let $U : \cL \to \R$ be a weight such that $\log U$ derives from the H\"older potential $W$, with $\wp(W) > 0$. Then,
$$\sum_{\gamma \in {\cL}_{T}} U(\gamma)\sim\frac{e^{\wp(W)T} }{\wp(W)T} \mbox{ when }T\rightarrow \infty.$$ 
\end{theorem}

Finally, we obtain the following meromorphic extension for the particular case of Guillarmou-Naud's weighted zeta function \cite{gn}, by relating it to a weighted zeta function which satisfies the hypotheses of the previous theorem. 

Let $W_{SBR}$ denote the Sinai-Bowen-Ruelle potential of the geodesic flow, and for $\gamma \in \cL$, let $\cP^k _{\gamma}$ denote the $k$-times Poincar\'e map about $\gamma$. 
\begin{theorem}\label{theo:WZeta2}
Let $(M, g)$ be a conformally compact manifold whose sectional curvatures satisfy $-b^2\leq K_g\leq -a^2<0$, with non Abelian fundamental group and non-arithmetic length spectrum. The weighted Zeta function
$$ \tilde{Z} (s) = \exp \left( \sum_{\gamma \in \cL_{p}} \sum_{k \in \N} \frac{e^{-k s l_p(\gamma)}}{k \sqrt{| \det(I - \cP^k _{\gamma}) |}} \right)$$
is an analytic non-zero function on the half-plane $\mathfrak{R}(s)> \wp(-\frac{W_{SBR}}{2})$. It admits a meromorphic extension to the half plane 
$$\mathfrak{R}(s)> \wp(-\frac{W_{SBR}}{2}) - \inf \left\{ \frac{\lambda a}{b}, \frac{\lambda}{2} \right\},$$
where $\lambda$ is the expansion factor of the geodesic flow on the non-wandering set. Moreover, with the exception of a simple pole at $\wp(-\frac{W_{SBR}}{2})$, this extension is analytic and non-vanishing in an open neighborhood of $\{ \mathfrak{R}(s) \geq \wp(-\frac{W_{SBR}}{2}) \}$.  If $\wp(-\frac{W_{SRB}}{2} ) > 0$, then we have the counting estimate
$$\sum_{\gamma \in \cL_T} | \det(I-\cP_{\gamma})|^{-\frac{1}{2}}\sim \frac{ \exp \left( \wp(-\frac{W_{SRB}}{2})T)\right)}{\wp(-\frac{W_{SRB}}{2} ) T} \mbox{ when } T\rightarrow\infty.$$
\end{theorem}

The proof suggests that this meromorphic extension to the half plane $\mathfrak{R}(s)> \wp(-\frac{W_{SBR}}{2}) - \inf \left\{ \frac{\lambda a}{b}, \frac{\lambda}{2} \right\},$ may also be optimal.   We wish to emphasize that these results are proven using only dynamical methods. In particular, we require neither pseudo-differential analysis nor semi-classical methods.
\medskip

In the final section of the paper \S 5, we explore connections between the dynamics and spectral theory on convex co-compact, conformally compact and asymptotically hyperbolic manifolds. We show the following relationship between the dynamics and the bottom of the spectrum of the Laplacian on convex co-compact manifolds with variable curvature.

\begin{theorem}\label{th:Eigenvalue} 
Let $(M, g)$ be a convex co-compact manifold of dimension $n+1$ whose sectional curvatures satisfy $-b^2 \leq K_g \leq -a^2<0$ and with non-Abelian fundamental group.
Let $h_g$ be the topological entropy of the geodesic flow on $S^gM$ restricted to its non-wandering set. If $h_g > \frac{na}{2}$, then the bottom of the spectrum $\lambda_0(g)$ of the Laplacian $\Delta_g$ satisfies
$$ h_g(na-h_g) \leq \lambda_0(g) \leq \frac{(nb)^2}{4} .$$
If $h_g\leq \frac{na}{2},$ then we have
$$ \frac{(na)^2}{4} \leq \lambda_0(g) \leq \frac{(nb)^2}{4} .$$
\end{theorem}

It can be hoped that there exist further relationships between the dynamics and the spectral theory on conformally compact manifolds, similar to what has been proven for hyperbolic manifolds. In particular, Theorem \ref{theo:WZeta2} suggests that there may be a relation between the Sinai-Bowen-Ruelle potential of the geodesic flow and the existence of pure-point spectrum. This paper, which focuses on the dynamics on such manifolds, is meant to be a first step in this direction. 

\subsection*{Acknowledgements}
All three authors would like to thank the Hausdorff Center for Mathematics and the Max Planck Institute for Mathematics, in Bonn (Germany), for the opportunity which they gave us to start this collaboration. We also acknowledge the support of the F\'ed\'eration Math\'ematiques des Pays de la Loire through its program MATPYL during the final stages of this work and insightful conversations with Gilles Carron, Laurent Guillop\'e and Rafe Mazzeo.  The second two authors gratefully acknowledge the support of the Laboratoire International Associe Salomon Lefschetz CNRS-CONACyT.

\section{Conformally compact and convex co-compact manifolds}\label{sec:Ccpct}
In this section, we recall the definitions of conformally compact and convex co-compact manifolds and demonstrate connections between these two notions. 

\begin{defn}\label{def:ConfCpct}
Let $M$ be the interior of a smooth closed (n+1)-manifold with smooth closed $n$-boundary $\pa M$. Let $g$ be a complete metric on $M$. We say that $g$ is \emph{conformally compact} if there is a smooth positive function $x : M \to (0,\infty)$ such that
\begin{enumerate}
\item $x$ vanishes to first order on the boundary
$$\lim_{p\to \pa M}x(p) = 0 \mbox{ and } \lim_{p\to \pa M}dx(p)\neq 0;$$
\item the (incomplete) metric $\bar g = x^2 g$
extends to a smooth metric, which we will still write $\bar g$, on $\bar{M} = M\cup \pa M$.
\end{enumerate}
\end{defn}
Such a smooth function $x$ defined at least in a neighborhood $\cN$ of $\pa M$ so that $x: \cN \to [0, \infty)$, $\pa M = x^{-1} (\{0\}),$ and $dx \neq 0$ on $\pa M$ is called a \em boundary defining function\em. If $(M,g)$ is conformally compact, and $x^2g$ is a smooth metric on $\bar M$, we will call $\pa M$ equipped with the induced conformal family of metrics the \emph{conformal boundary} of $M$. We define a \em conformally compact manifold \em to be a complete Riemannian manifold $M$ equipped with a conformally compact metric $g$. Conformally compact manifolds have strictly negative sectional curvatures at infinity. More precisely, we have the following. 

\begin{prop}\label{prop:SecInfty}
Let $(M,g)$ be a complete metric on a compact manifold $M$ with non-empty boundary $\pa M$. Assume that $g$ is conformally compact, and let $x : M\to (0,\infty)$ be a boundary defining function such that $\bar g = x^2g$ is a smooth metric on $\bar M = M\cup \pa M$. Then for any point $p_\infty\in \pa M$, we have
$$\lim_{p\to p_\infty}K_g(p) = -\abs{dx}^2_{\bar g}(p_\infty),$$
where $K_g(p)$ is any of the sectional curvatures of $(M,g)$ at the point $p$.
\end{prop}
\begin{proof}
The formul{\ae}\, given in \cite{Bes} p.58 to compute the $(4,0)$-curvature tensor of $g = \frac{\bar g}{x^2}$ imply the asymptotic expansion
$$R_g = -\frac{1}{x^4}\bar g\odot |dx|^2_{\bar g}\bar g + \mathcal O\left(\frac{1}{x^3}\right)$$
when $x\rightarrow 0$, where $\odot$ is the Kulkarni-Nomizu product. Therefore, the sectional curvatures are asymptotic modulo $\mathcal O(x)$ to the function $-|dx|^2 _{\bar g}$ on $\pa X$ as $x \to 0$.
\end{proof}

The simplest example of a conformally compact manifold is the real hyperbolic $n$-space, which is conformally equivalent to the Euclidean metric on the unit ball. The next natural set of examples are infinite volume \emph{convex co-compact} hyperbolic manifolds, which are quotients of hyperbolic $n$-space by a discrete group, such that in the quotient all closed geodesics remain in a compact set.  Another special set of examples of conformally compact manifolds are the \emph{asymptotically hyperbolic manifolds}, whose definition is originally due to Mazzeo and Melrose (see \cite{mmah}). We will now describe the larger class of \emph{convex co-compact manifolds with variable (negative) curvature} and relate them to conformally compact manifolds. 

Recall that a \emph{Hadamard manifold} $(M,g)$ is a simply connected $n$-manifold with non-positive sectional curvatures. The universal cover of any Riemannian manifold with negative sectional curvatures is a Hadamard manifold. The \emph{visual boundary}, which we will note $\pa_v M$, is defined to be the set of equivalence classes of geodesic rays, where two rays are equivalent when they remain at finite distance from each other. There is a natural topology on $\bar{M}_v = M\cup \pa_v M$, described as the \emph{cone topology} in \cite{e72} p. 495. Equipped with this topology, $\bar{M}_v$ is homeomorphic to the closed unit ball in $\mathbb{R}^n$. Any isometry acting on $\tilde{M}$ extends continuously to an action on the visual boundary $\pa \tilde{M}$. Let $\Gamma$ be the fundamental group of $M$ acting on $\tilde{M}$ by isometries. We denote by $\Lambda_\Gamma$ its \emph{limit set}, which is the smallest closed non-empty $\Gamma$-invariant subset in $\pa \tilde{M}$. It is also the intersection with $\pa \tilde{M}$ of the closure of the orbit by $\Gamma$ of any point in $\tilde{M}$. It follows from the Proposition 2.6 of \cite{e72} that the end points in $\pa_v\tilde M$ of lifts of closed geodesics in $M$ are dense in $\Lambda_\Gamma$. We write $CH(\Lambda_\Gamma)$ for the convex hull of $\Lambda_\Gamma$ in $\tilde{M}$. The set $$C(M) := CH(\Lambda_\Gamma)/\Gamma$$
is called the \emph{convex core} of $M$. 

The visual boundary of a Hadamard manifold $M$ is sometimes also called its ``conformal boundary.'' Actually, the following proposition shows that the concept of visual boundary generalizes the conformal boundary to all Hadamard manifolds. 

\begin{prop}
Let $(M,g)$ be a conformally compact simply connected manifold with negative sectional curvatures. Then there is a canonical homeomorphism $I_v$ between the compact spaces $\bar M = M\cup \pa M$ and $\bar{M}_v = M\cup \pa_v M$. 
\end{prop}
\begin{proof}
Let $\bar g = x^2 g$ be a smooth metric on $\bar M = M\cup \pa M$, where $x^2$ is a boundary defining function for $\pa M$. Let $O\in M$ be a fixed origin point. Then, any vector $X\in T_O M$ defines a unique point $\xi_v(X)=[\gamma_X]\in \pa_v M$ which is the equivalence class of the geodesic ray $\gamma_X$ which starts at $O$ with tangent vector $X$. Let $\xi_c(X)\in \pa M$ be the end point of $\gamma_X$ in $\pa M$. The next lemma follows immediately from the definition of $x$ as a boundary defining function. 
\begin{lemma}
Let $O$, $O'$ be two points of $M$, and $X\in T_OM$, $X'\in T_{O'}M$ such that the distance between the geodesic rays $\gamma_X$ and $\gamma_{X'}$ remains bounded up to $\pa M$. Then $d_{\bar g}(\gamma_X(t),\gamma_{X'}(t))\to 0$ as $t\to\infty$.
\end{lemma}
This ensures that if $\gamma_{X'}\in [\gamma_X]$, then $\gamma_X$ and $\gamma_{X'}$ end at the same point of $\pa M$. Therefore the map 
$$I_v : \xi_v = \xi_v(X)\in \pa_v M\mapsto \xi_c(X)$$
is well defined and invertible. It follows from Proposition 1.3 of \cite{e72} that $I_v$ extends the identity on $M$ to a homeomorphism $I_v : \bar{M}_v = M\cup \pa_v M \to \bar M = M\cup \pa M.$\end{proof}

In the following, we will use $\pa M$ to denote both the visual and the conformal boundary of simply connected manifolds, which will be implicitly identified by the above homeomorphism. We introduce few more definitions relevant to negatively curved manifolds with non-trivial topology.

\begin{defn} 
A Riemannian manifold $(M,g)$ has \emph{pinched negative curvature} if there exist constants $0 <a \leq b$ such that the sectional curvatures $K_g$ of $M$ satisfy
$$-b^2 \leq K_g\leq -a^2 <0.$$
\end{defn} 

When $(M,g)$ is convex co-compact (which will always mean in this paper that it has pinched negative curvature), it follows from the preceding proof that the convex core $C(M)$ is also the intersection of all convex subsets $C\subset M$ whose interiors $\stackrel{\circ}{C}$ are homeomorphic to $M$. 

\begin{defn}
The set $R_\Gamma = \pa \tilde{M}\backslash \Lambda_\Gamma$ is called the \emph{regular set}. It is the maximal subset of $\pa \tilde M$ on which $\Gamma$ acts freely discontinuously. We will call $\pa_v M = R_\Gamma/\Gamma$ the \emph{visual boundary} of $M$.
\end{defn}

A smooth curve $\gamma$ in $M$ converges to a point $\xi \in \pa_v M$ if and only if each of its lifts $\tilde \gamma$ converge to a point $\tilde\xi\in R_\Gamma$ which is a lift of $\xi$. We call $\bar{M}_v = M\cup \pa_v M$ equiped with this topology the \emph{visual closure} of $M$. The following proposition defines convex-compactness, which will be a crucial concept in this paper. The name ``\emph{convex co-compact}" was first used by D. Sullivan \cite{Sul79} for hyperbolic manifolds. Convex co-compactness is strongly related to the notion of \emph{geometric finiteness}, which has also been introduced for hyperbolic 3-manifolds. Geometrically finite manifolds of pinched negative curvature have been defined in detail by B. Bowditch in \cite{Bow95}; we shall not give here a complete definition of this notion. We will only be interested in the following.

\begin{prop}[Convex-cocompactness \cite{Bow95}]\label{prop:CvCpct}
Let $(M,g)$ be a complete Riemannian manifold with pinched negative curvature. The following assertions are equivalent:
\begin{enumerate}
\item the manifold $M$ is geometrically finite without parabolic cusp in the sense of \cite{Bow95}; 
\item the visual closure $\bar{M}_v = M\cup \pa_v M$ of $M$ is compact.
\item the convex core $C(M)$ is compact; 
\item there exists a \emph{convex compact subset} $K\subset M$ such that all closed geodesics of $M$ are contained in $K$; 
\item there exists a \emph{convex compact subset} $K\subset M$ whose interior $\stackrel{\circ}{K}$ is homeomorphic to $M$; 
\end{enumerate}
The manifold $(M,g)$ is said to be \emph{convex co-compact} if it satisfies one of these properties.
\end{prop}
This result is well known, even though we could not locate it in the literature. It follows directly from the arguments of \cite{Bow95}. We include a proof for completeness.
\begin{proof}
The equivalence between (1), (2) and (3) is exactly given by Theorem 6.1 of \cite{Bow95}. We will prove $(5)\Rightarrow (4)\Rightarrow (3)$ and $(2)\Rightarrow (5)$.
\begin{description}
\item[$(5)\Rightarrow (4)$]
Assume that there exists a convex compact subset $K\subset M$ whose interior $\stackrel{\circ}{K}$ is homeomorphic to $M$. Let $\gamma\subset M$ be a closed geodesic. Since $\stackrel{\circ}{K}$ is homeomorphic to $M$, there is a closed curve $\alpha\subset \stackrel{\circ}{K}$ which represents the homotopy class of $\gamma$. 
Since $K$ is convex, there is a closed geodesic $\gamma'\subset K$ which is homotopic to $\alpha$. Since in a negatively curved manifold, each homotopy class contains a unique closed geodesic, we have $\gamma = \gamma'\subset K$.
\item[$(4)\Rightarrow (3)$]
Let $\Gamma$ be the fundamental group of $M$ acting on $\tilde{M}$ by isometries and $\Lambda_\Gamma$ its limit set. 
Let $K$ be a convex compact set which contains all closed geodesics. Therefore, the lift $\tilde K\subset \tilde M$ of $K$ to the universal cover is a $\Gamma$-invariant convex subset of $\tilde M$, and its closure for the visual topology contains all points on $\pa_v \tilde M$ which are end points of lifts of closed geodesics. Therefore since these points are dense in $\Lambda_\Gamma$, and since $\tilde K$ is convex, it contains $CH(\Lambda_\Gamma)$. This implies that the convex core $C(M)$ is contained in $K$, and is therefore compact.
\item[$(2)\Rightarrow(5)$]
Let $M$ be a complete $n$-manifold with negatively pinched curvature such that its visual closure $\bar{M}_v = M\cup \pa_v M$ is compact. We denote by $\Gamma$ the fundamental group of $M$ acting on $\tilde{M}$ by isometries. There exists a compact $K_1\subset M$ which is homeomorphic to $\bar{M}_v$. Let $\tilde K_1\subset \tilde M$ be the lift of $K_1$ to the universal cover; $\tilde K_1$ is homeomorphic to the closed unit ball of $\R^{n+1}$ via a $\Gamma$-equivariant homeomorphism. Hence its convex-hull $\tilde K_2 = CH(\tilde K_1)$ is also homeomorphic to the closed unit ball of $\R^{n+1}$ via a $\Gamma$-equivariant homeomorphism. Moreover, by Proposition 2.5.4 of \cite{Bow95}, $\tilde K_2$ is contained in an $r$-neighbourhood of $\tilde K_1$, where $r>0$ only depends on the upper bound on the sectional curvatures. Therefore, $K_2 = \tilde K_2/\Gamma$ is a convex set, homeomorphic to $M$, which is contained in the $r$-neighbourhood of $K_1$. Hence $K_2$ is compact, which concludes our proof.\end{description}\end{proof}

To prove our results, we will require that our manifolds have enough complexity: their fundamental groups should not be generated by the class of a single geodesic. This property can be stated in many ways for convex co-compact manifolds.

\begin{prop}\label{prop:Abe1}
Let $(M,g)$ be a complete convex co-compact Riemannian manifold with pinched negative curvature, and $\Gamma = \pi_1(M)$. Then the following are equivalent: 
\begin{enumerate}
\item $\Gamma$ is abelian;
\item the limit set $\Lambda_\Gamma$ is finite;
\item the cardinality of $\Lambda_\Gamma$ is $2$; 
\item $\Gamma =\pi_1(M)$ is generated by the class of a single closed geodesic.
\end{enumerate}
\end{prop}
\begin{proof}
Let $(M,g)$ be a complete convex co-compact Riemannian manifold with pinched negative curvature. The equivalence $(1)\Leftrightarrow (2)$ is valid for any manifold with negative curvature. Indeed, if $\Gamma$ is abelian, the fixed points (at most two) of any element of $\Gamma$ are invariant by $\Gamma$, therefore the limit set is contained in these points. The converse comes immediately from the classification of isometries for negatively curved spaces, cf \cite{Bow95}.

Since $(M,g)$ is convex co-compact, $\Gamma =\pi_1(M)$ can not contain any parabolic element. Therefore, $(2)\Leftrightarrow(3)$. 

Finally, since $\Gamma = \pi_1(M)$ acts discretely on the universal cover of $M$, we have $(3)\Leftrightarrow(4)$.
\end{proof}

We will now show that for manifolds with constant negative curvature, convex-compact and conformally compact are equivalent notions. 

\begin{prop}\label{prop:HypConfCpct}
Let $(M,g)$ be a real hyperbolic $n$ dimensional manifold. Then $g$ is conformally compact if and only if $(M,g)$ is convex co-compact.
\end{prop}
\begin{proof}
First, we let $(M,g) = H^n/\Gamma$ be a convex co-compact hyperbolic manifold. By Proposition \ref{prop:CvCpct}, its visual closure is compact. Therefore, the construction of Paragraph 3.1 of \cite{pp} applies and shows that $(M,g)$ is conformally compact. 
We will prove in Theorem \ref{th:ConfCvxCpct} below that a conformally compact manifold with pinched negative curvature is always convex co-compact, which ends the proof of this proposition.
\end{proof}

For \em surfaces \em with variable negative curvature, convex co-compact and conformally compact are actually equivalent notions:

\begin{prop}\label{prop:surface}
Let $(S,g)$ be a complete (open) Riemannian surface with pinched negative curvature. Then $g$ is conformally compact if and only if $(S,g)$ is convex co-compact.
\end{prop}

\begin{proof}
Let $(S,g)$ be a complete open Riemannian surface with boundary and pinched negative curvature. Then, there exists a unique hyperbolic metric $g_H$ on $S$ conformally equivalent to $g$, i.e. $g_H = u^2g$, where $u$ is a positive function on $M$. Since the curvature of $(S,g)$ is bounded, it follows from Yau's Schwarz Lemma (cf \cite{Yau73}, Theorem 1 p. 373) that $u$ is pinched between positive constants. Then, if $x$ is a boundary defining function, $x/u$ and $ux$ are also boundary defining functions, so it follows directly from the definition that $(S, g)$ is conformally compact if and only if $(S, g_H)$ is. 

Thus, we are reduced to showing that a complete, open hyperbolic surface $(S, g_H)$ is conformally compact if and only if it is convex co-compact, but this follows from the preceding proposition. 
\end{proof}

The case of manifolds with arbitrary dimension and variable negative curvature is less understood. We first observe the following useful relationship between the two notions conformally compact and convex co-compact for manifolds of arbitrary dimension. Although we expect this result is known, we have not seen a proof in the literature. 

\textbf{Proof of Theorem \ref{th:ConfCvxCpct}}
Let $(M,g)$be a conformally compact manifold with negative sectional curvatures. It follows from Proposition \ref{prop:SecInfty} that $(M,g)$ has pinched negative curvature. Let $x: M \to (0, \infty)$ be a boundary defining function. It follows from Corollary 5 of Bahuaud \cite{ba09} that for $\epsilon > 0$ sufficiently small, 
$$K_{\epsilon} := \left\{ x^{-1} [\epsilon, \infty) \right\}$$
contains all closed geodesics of $M$. Therefore, the first assertion of the theorem then follows from Proposition \ref{prop:CvCpct} (4). In case $M$ is a surface or if $g$ has constant sectional curvatures, the converse statement follows from Propositions \ref{prop:HypConfCpct} and \ref{prop:surface}. 
\qed

\begin{remark}
By Proposition \ref{prop:SecInfty}, the sectional curvatures of any conformally compact manifold converge to a limit at each point of the boundary. This is obviously \emph{not} satisfied by all convex co-compact manifolds with negative curvature. One of the simplest examples is given by the complex hyperbolic space $\bbH_\C^n$, which is homeomorphic to the interior of the unit ball. Its canonical metric is Einstein, and at any point, the sectional curvatures span the interval $[-4, -1]$. Thus, there is no scalar function to which the sectional curvatures converge at the visual boundary. 
Any convex co-compact manifolds which are also complex hyperbolic (quotients of $\bbH_\C^n$) do not satisfy Proposition \ref{prop:SecInfty} and are not conformally compact.
\end{remark} 

We conclude this section with physical motivation to study conformally compact manifolds. Recall that a Riemannian metric is \em Einstein \em if the Ricci and metric tensors satisfy the relationship 
$$\textrm{Ric} = c g,$$
for some constant $c$. A metric which is both conformally compact and Einstein, is called a \emph{Poincar\'e Einstein metric}; it follows from Proposition \ref{prop:SecInfty} that $c<0$. Examples of the Poincar\'e Einstein metrics which arise in AdS/CFT correspondence in string theory include the hyperbolic analogue of the Schwarzschild metric, and in dimension $4$, the Taub-BOLT metrics on disk bundles over $\bbS^2$. For these and further examples, (see M. Anderson \cite{and2}).

\section{The geodesic flow}\label{sec:GeoFlow}
Let $(M,g)$ be a complete Riemannian manifold with unit tangent bundle $S^gM$, and let $S^g_xM$ denote its fiber over a point $x\in M$. We will omit the metric $g$ when no confusion shall arise. The \emph{geodesic flow} on $(M,g)$ is the map $\Phi^g : S^gM \times \reals\to S^gM$ which maps $v\in S_x M$ and $t\in\reals$ to 
$$\Phi_t(v) = \dot{\gamma}_v(t),$$
where $\gamma_v$ is the geodesic starting from $x$ with initial tangent vector $v$, and $\dot{\gamma}_v(t)$ is its tangent vector at the point $\gamma_v(t)$. The \emph{orbit} of a vector $v\in S_xM$ is $\{ \Phi_t (v) | t \in \reals \}$. The restriction of the canonical projection to the orbit of $v$ is bijective onto the geodesic $\gamma_v \subset M$. Therefore, we will often identify an orbit of the flow with the geodesic in $M$ onto which it projects. 

\begin{defn}
Let $(M,g)$ be a complete Riemannian manifold. A vector $v\in S^gM$ is \emph{non-wandering} for the geodesic flow $\Phi^g$ if for all neighbourhoods $\cO$ of $v$ in $SM$, there exists a sequence $(t_n)\in\reals^{\N}$ with $t_n\to+\infty$ when $n\to+\infty$ such that for all $n\in\N$, 
$$\cO\cap \Phi^g_{t_n}\cO \neq \emptyset.$$
The set of non-wandering vectors of $SM$ (for the geodesic flow) denoted by $\Omega(g)$ is called the \emph{non-wandering set}.
\end{defn}

Let $\Omega\subset SM$ be a closed $\Phi_t$-invariant set. The flow is called \emph{uniformly hyperbolic} on $\Omega$ if there exists $\lambda>0$ such that for each $v \in \Omega$, $T(SM)_{v}$ splits into a direct sum 
$$ T(SM)_{v} = E^s _{v} \oplus E^u _{v} \oplus E_{v},$$
such that $E_v$ is the tangent space of $\{\Phi_t(v)\}_{t\in\mathbb{R}}$, and for all $t\geq 0$,
\begin{eqnarray}\label{eq:HypFlow}
\norm{d\Phi_t(\xi)} & \leq e^{-\lambda t}\norm{\xi} & \mbox{ if } \xi\in E^s _{v}\\
\norm{d\Phi_{-t}(\xi)} & \leq e^{-\lambda t}\norm{\xi} & \mbox{ if } \xi\in E^u _{v}.
\end{eqnarray}
The constant $\lambda$ is called the \emph{expansion factor} of the flow on $\Omega$.

These definitions are \em local, \em depending only on properties of the flow at a point or restricted to a closed set, respectively. We shall see that certain aspects of the dynamics of compact manifolds which only depend on these local conditions can be extended to infinite volume manifolds. 
For example, the hyperbolicity factor $\lambda$ is related to the curvature pinching constants as follows. If the sectional curvatures $K_g$ satify $-b^2 \leq K_g \leq - a^2$ on a $\Phi_t$-invariant subset $\Omega\subset SM$, then the geodesic flow is uniformly hyperbolic on $\Omega$, with expansion satisfying $a \leq \lambda \leq b$. This follows from Theorem 3.9.1 of \cite{Kli82} p.364, which holds for complete manifolds, so we can use it for convex co-compact manifolds. 

The following definition lists properties of the geodesic flow on convex co-compact manifolds which are crucial for our arguments. 

\begin{defn}
Let $(M,g)$ be a Riemannian manifold and $(\Phi_t)_{t\in\R}$ be its geodesic flow.
\begin{enumerate}
\item The geodesic flow satisfies S. Smale's \em Axiom A \em \cite{sm} if its non-wandering set $\Omega(g)$ is compact, the flow is uniformly hyperbolic on $\Omega(g)$ and the periodic orbits of the flow are dense in $\Omega(g)$. 

\item Let $\mathcal B\subset SM$ be a closed set of the unit tangent bundle. The flow is \em topologically transitive \em on $\mathcal B$ if for any open $U$ and $V \subset\mathcal B$, there exists $t > 0$ such that $\Phi_t (U) \cap V \neq \emptyset$. 

\item The flow is \emph{topologically mixing} on $\mathcal B$ if for any open $U$ and $V \subset \mathcal B$, there exists $T > 0$ such that for all $t>T$, $\Phi_t (U) \cap V \neq \emptyset$. 
\end{enumerate}
\end{defn}

We emphasize the fact that we do not require the flow to be topologically transitive or mixing on the whole tangent bundle, but only \emph{when restricted to a closed subset}, which in our applications will be the non-wandering set. 

A \em basic hyperbolic set \em $\mathcal B\subset \Omega$ of an Axiom A flow $\Phi_t$ is a closed subset of the non-wandering set $\Omega$ such that $\Phi_t |\mathcal B$ is topologically transitive. This definition of \em basic set \em is due to R. Bowen \cite{Bow73}.

Some of our proofs rely on the following observation. 

\begin{prop}[Eberlein]\label{prop:AxiomA}
Let $(M,g)$ be a convex co-compact manifold with pinched negative curvature whose fundamental group is not Abelian. Then its geodesic flow is an Axiom A flow whose unique basic hyperbolic set is the non-wandering set $\Omega(g)$. Moreover, the flow is topologically transitive on $\Omega(g)$.
\end{prop}

Eberlein proved in Theorem 6.2 of \cite{e73} that if $(M,g)$ is \emph{any} complete manifold with pinched negative curvature and non-Abelian fundamental group, and if there exists two closed geodesics $\gamma, \gamma'$ such that $\ell(\gamma)/\ell(\gamma')\notin\bbQ$, then the geodesic flow is topologically mixing on its non-wandering set. A complete characterization was obtained in\cite{Dal}.

\begin{theorem}[Dalb'o, 2000]\label{theo:Dalbo}
Let $(M,g)$ be a complete Riemannian manifold with pinched negative curvature and non-Abelian fundamental group. Then the geodesic flow is topologically mixing on its non-wandering set \emph{if and only if the length spectrum is non-arithmetic}, i.e. if the lengths of all closed geodesics generates a dense subgroup of $\R$.
\end{theorem}

It has been shown that the length spectrum of a complete manifold $(M,g)$ with pinched negative curvature and non-Abelian fundamental group is non-arithmetic when: either $M$ is compact; or $M$ is a surface; or $M$ has constant sectional curvatures;
or if the fundamental group of $M$ contains a parabolic element. 

The references for these results are given in \cite{Dal}, p 982. Therefore, for a general convex co-compact manifold with pinched negative curvature and non-Abelian fundamental group, it is in general unknown whether the geodesic flow is topologically mixing or not. \em The authors thank B. Schapira for pointing out that the geodesic flow is still not known to be mixing on convex co-compact manifolds with variable curvature, even though most experts believe that it should be. \em 

\subsection{The critical exponent, entropy and pressure}

Let $(M,g)$ be a complete manifold with pinched negative curvature, with fundamental group $\Gamma$ acting by isometries on the universal cover $(\tilde M,g)$. Given two points $x,y\in \tilde M$, the \emph{Poincar\'e series of $\Gamma$} (in $x$ and $y$) is 
$$P(x,y,s) := \sum_{\gamma \in \Gamma} \exp\left( - s d_g(x, \gamma y) \right),$$
where $d_g$ is the Riemannian distance induced by $g$. The critical exponent $\delta(\Gamma)$ is defined so that the Poincar\'e series converges for $s > \delta (\Gamma)$ and diverges for $s < \delta (\Gamma)$. It can be easily checked that the critical exponent is well defined and does not depend of the base points $x$ and $y$.

This critical exponent is strongly related to the \emph{topological entropy} of the geodesic flow. Let us recall its definition: for any large $T>0$ and small $\delta > 0,$ a finite set $Y \subset SX$ is called $(T, \delta)$ separated if, given $\xi, \xi' \in Y, \xi \neq \xi',$ there is $t \in [0, T]$ with $d(\Phi_t \xi, \Phi_t \xi') \geq \delta.$ Here the distance on $SX$ is given by the Sasaki metric. 
\begin{defn} Let $(X, g)$ be convex co-compact with pinched negative curvature. Let $\Omega \subset SX$ be the non-wandering set of the geodesic flow. The \em topological entropy of the geodesic flow \em is 
$$h(g) := \lim_{\delta \to 0} \limsup_{T \to \infty} \frac{ \log \sup \# \{ Y \subset \Omega: Y \textrm{ {\rm is} $(T, \delta)$ {\rm separated}} \}}{T}.$$
\end{defn} 

Otal and Peign\'e proved in \cite{Ot-Pei} the following.

\begin{theorem}[Otal-Peign\'e, 04]
Let $(M,g) = (\tilde M,g)/ \Gamma$ be a convex co-compact manifold with pinched negative curvature. Then the critical exponent of the Poincar\'e series of $\Gamma$ and the topological entropy of the geodesic flow restricted to the non-wandering set coincide.
\end{theorem}

Actually, Otal and Peign\'e proved that topological entropy and critical exponent coincide for general complete manifolds with pinched negative curvature. The definition of the topological entropy has to be slightly adapted when the non-wandering set is non-compact. We will not do it here, since we have seen in Theorem \ref{th:ConfCvxCpct} that the non-wandering set of a conformally compact manifold with negative curvature is always compact. 

Given a homeomorphism $\phi : X\rightarrow X$ on a compact metric space, and a probability measure $\mu$ which is invariant by $\phi$, one can define the \emph{metric entropy of $\mu$ with respect to $\phi$}, denoted by $h_\mu(\phi)$. The standard \emph{variational principle} (cf \cite{Bow75}, Section 2) states that the topological entropy $h_{top}(\phi)$ of $\phi$ is the supremum of the metric entropies for all probability measures invariant by $\phi$. We will also use the notion of \emph{pressure}, which generalizes the topological entropy defined above.

\begin{defn} Let $X$ be a compact metric space, $\phi$ a homeomorphism on $X$ and $f : X\rightarrow \R$ a H\"older map. The \emph{pressure} of $f$ with respect to $\phi$ is
$$\wp(f) = \sup_\mu \{h_\mu(\phi)+\int_{X} fd\mu\},$$
where $\mu$ runs over all probability measures on $X$ invariant under $\phi$, and $h_{\mu}$ is the measure theoretic entropy of $\phi$ with respect to $\mu$. When $(\Phi_t)_{t\in\R}$ is a flow on $X$, the 
pressure of any H\"older map $f:X\rightarrow \R$ is the pressure of $f$ with respect to $\phi = \Phi_1$.
\end{defn} 
In particular, the variational principle mentioned above states that the topological entropy
$$h_{top}(\phi) = \wp(0).$$
A detailed description of the concepts of entropy and pressure, together with its dynamical applications, can be found in \cite{Bow75}.

In the special case of the geodesic flow on a convex co-compact manifold with pinched negative curvature, the non-wandering set is a compact set invariant under the flow, and the entropy is equivalently given by the length spectrum asymptotics. 

\begin{prop}\label{prop:pot} 
Let $(M,g)$ be a convex co-compact manifold with pinched negative curvature, $(\Phi_t)$ its geodesic flow, $\gL$ the set of its closed geodesics and $\Omega(g)$ be its non-wandering set.
\begin{enumerate}
\item The topological entropy $h(g)$ of $\Phi$ restricted to $\Omega(g)$ is given by
$$h(g) = \lim_{T\rightarrow \infty} \frac{\log\#\left\{\gamma\in \gL ; \ell_g(\gamma)\leq T\right\}}{T},$$
where $\gL$ is the set of all closed geodesics of $(M,g)$ and $\ell_g$ is the length induced by $g$ on $\gL$.
\item With this notation, $h(g) = 0$ if and only if $\pi_1(M)$ is Abelian.
\end{enumerate}
\end{prop}

\begin{proof}
Let $(M,g)$ satisfy the hypotheses of the proposition. We have seen in the previous section that the geodesic flow on $(M,g)$ is Axiom A, with $\Omega(g)$ its unique basic set. Therefore, the first assertion is a direct application of Theorem 4.11 of \cite{Bow73}. 

If $\pi_1(M)$ is Abelian, then it follows from Proposition \ref{prop:Abe1} that it is generated by the homotopy class of a single closed geodesic of length $\ell$. Therefore, $$\#\left\{\gamma\in \gL ; \ell_g(\gamma)\leq T\right\}\sim \ell T \mbox{ when } T\rightarrow\infty,$$ 
where we have used the notation $f(t) \sim g(t) \quad \textrm{when} \quad t \to \infty \iff \lim_{t \to \infty} \frac{f(t)}{g(t)} = 1$ whenever $f, g: \R \to \R$ with $g(t) > 0$ for all $t$ sufficiently large. By item (1), this implies that $h(g) = 0$. 

If $\pi_1(M)$ is not Abelian, then it follows from Proposition \ref{prop:Abe1} that its limit set is infinite. This implies that $\pi_1(M)$ contains two hyperbolic elements $\gamma_1, \gamma_2$ with distinct fixed points on the visual boundary which are not conjugate within $\pi_1(M)$, and hence they generate a free group $F_2 = \langle \gamma_1, \gamma_2\rangle$. Moreover, since $M$ is convex co-compact, all closed geodesics are contained in a compact set $K$. Therefore, there exists a $l_{min}>0$ such that all closed geodesics have length at least $l_{min}$. Let $d$ be the diameter of $K$ and $\gamma$ be a closed geodesic whose homotopy class $\gamma\in F_2$, and such that $\gamma$ can be expressed by a word of in $\gamma_1, \gamma_2$ of length $T$. By lifting to the universal cover, one can see that the length of $\gamma$ satisfies $\ell(\gamma)\leq (l_{min}+2d)T$. Since each conjugacy class of the fundamental group corresponds exactly to one closed geodesic, this implies
$$\#\left\{\gamma\in \gL ; \ell_g(\gamma)\leq (l_{min}+2d)T\right\} \geq \mathcal{O}(3^T).$$
Therefore, $h(g)>0$ by item (1).
\end{proof}

\section{Zeta functions and asymptotic counting estimates}
S. Smale introduced the Zeta function associated to a general dynamical system \cite{sm} with the idea that Selberg's techniques from analytic number theory \cite{sel} could be applied to dynamical settings. Bowen \cite{Bow73} made significant contributions through his work in symbolic dynamics, which encodes the dynamics of the geodesic flow or a more general Axiom A flow into somewhat easier to manage symbolic dynamics. If the flow is mixing, then the zeta function can be written in terms of symbolic zeta functions. This is one of the main technical tools used by Pollicott \cite{Pol86}, Haydn \cite{Hay90} and several other authors to study dynamical zeta functions and asymptotic counting estimates for the number of closed geodesics. 

The dynamical Zeta function is also a crucial way to link the closed geodesics of hyperbolic manifolds with its spectral theory. Perry \cite{plength} used a spectral interpretation of the poles of the dynamical zeta function to prove the prime orbit theorem on convex co-compact hyperbolic manifolds in dimension $n$. By considering a suitably \emph{weighted dynamical Zeta function}, Guillarmou and Naud \cite{gn} generalized Perry's result by producing \em a larger asymptotic expansion \em for the geodesic length counting function, 
which is explicitly determined by the pure point spectrum of the Laplacian. This result will be discussed in \S 5. 

In this paragraph, we will focus on counting estimates obtained from (weighted) dynamical zeta functions through purely dynamical methods. Most results which we use, especially from the works of Bowen, Parry, Pollicott and Haydn, rely on symbolic dynamics. Nevertheless, we have chosen to avoid a technical presentation of symbolic dynamics in this note. A complete and self-contained description of symbolic dynamics and its applications to zeta functions can be found in \cite{Par-Pol-Aster}.

\medskip

Let $\pgL$ be the set of primitive closed orbits of the geodesic flow and $l_p(\gamma)$ is the primitive period (or length) of $\gamma \in \pgL$. The weighted Zeta function of \cite{gn} is, 
$$ \tilde{Z} (s) = \exp \left( \sum_{\gamma \in \pgL} \sum_{k \in \N} \frac{e^{-k s l_p(\gamma)}}{k \sqrt{| \det(I - \cP^k _{\gamma}) |}} \right),$$
where $\cP^k _{\gamma}$ is the $k$-times Poincar\'e map of the geodesic flow around the primitive closed orbit $\gamma$. The behavior of this zeta function is governed by the \em Sinai-Bowen-Ruelle potential. \em 

\begin{defn}
For any $\xi\in S^1M$, we denote by $E^u(\xi)$ its unstable manifold. The \em Sinai-Bowen-Ruelle potential \em at $\xi$ is defined by 
$$W_{SBR}(\xi) : = \frac{d}{dt} \big|_{t=0} \log \det d\Phi_t |_{E^u _{\xi}},$$
where the determinant of $d\Phi_t|_{E^u _{\xi}} : E^u _{\xi}\rightarrow E^u_{\Phi_t\xi}$ is computed with respect to orthonormal bases of $E^u _{\xi}$ and $E^u_{\Phi_t\xi}$ for the Riemannian metric.
\end{defn}
This potential, introduced by Bowen and Ruelle (see \cite{BoRu75} section 4.), gives the instantaneous rate of expansion at $\xi$. Its regularity 
depends on the regularity of the unstable foliation, which is in general not smooth. However, 
when the sectional curvature is negatively pinched, then the unstable foliation is H\"older. 

\begin{prop}\label{prop:WSBR}
Let $M$ be a smooth complete $(n+1)$-manifold with pinched negative curvature: $-b^2\leq K_g\leq -a^2<0$. Then the Sinai-Bowen-Ruelle potential $W_{SBR}$ is a H\"older map on the unit tangent bundle $S^1M$. Its H\"older exponent can be taken to be
$\alpha =\inf\{\frac{2a}{b},1\}$. Moreover, the pressure of $W_{SBR}$ satisfies
$$h(g)-\frac{nb}{2}\leq \wp\left(-\frac{W_{SBR}}{2}\right)\leq h(g)-\frac{na}{2},$$
where $h(g)$ is the topological entropy of the geodesic flow. 
\end{prop}
\begin{proof}
Let $M$ be a complete $(n+1)$ dimensional manifold with pinched negative curvature: $-b^2\leq K_g\leq -a^2<0$. Since the flow $\Phi_t$ is smooth, the H\"older exponent of $W_{SBR}$ is the same as the H\"older exponent of the unstable foliation. It was shown in \cite{HirPug75} that when the curvature is $\frac{1}{4}-$pinched, i.e. $\frac{a}{b}>\frac{1}{2}$, then the unstable foliation is at least $\mathcal C^1$. Moreover, as observed in \cite{Has94}, it follows from the proof of Theorem 3.2.17 of \cite{Kli82} that when the curvature is not $\frac{1}{4}-$pinched, the unstable foliation is $\alpha-$H\"older where $\alpha$ can be taken to be $\frac{2a}{b}$, which concludes the proof of our first statement.

Since $W_{SBR}$ is H\"older, we can compute its pressure. It follows from Theorem 3.9.1 of \cite{Kli82} that
$$na\leq \frac{d}{dt} \big|_{t=0} \log \det d\Phi_t |_{E^u _{\xi}}\leq nb.$$
Therefore, for every probability measure $\mu$ on $S^1M$ invariant by the geodesic flow $\Phi$, we have
$$-\int_M \frac{nb }{2}d\mu +h_\mu(\Phi)\leq -\int_M \frac{W_{SBR}}{2} d\mu +h_\mu(\Phi)\leq -\int_M \frac{na }{2}d\mu +h_\mu(\Phi),$$
where $h_\mu(\Phi)$ is the entropy of the measure $\mu$ with respect to the geodesic flow. This implies immediately that
$$h(g)-\frac{nb}{2}\leq \wp\left(-\frac{W_{SBR}}{2}\right)\leq h(g)-\frac{na}{2},$$
where $h(g)$ is the topological entropy of the geodesic flow.
\end{proof}

We will call \emph{weight} a map $U : \gL\rightarrow \R$. A \emph{weighted dynamical Zeta function} (for the geodesic flow on $\Omega$) can be defined at least formally as
$$Z_U(s) = \exp \left( \sum_{\gamma \in \pgL} \sum_{k \in \N} \frac{1}{k}e^{-k s l_p(\gamma) + U(k\gamma)} \right)$$
for some weight $U : \gL\rightarrow \R$. For a special class of weights, we will show that such a weighted Zeta function converges to an analytic function on a half plane, and admits a meromorphic extension to a strictly larger half plane.

\begin{defn}
Let $U : \gL\rightarrow \R$ be a weight. We say that $U$ \emph{derives from a H\"older potential} if there is a H\"older map $W : \Omega\rightarrow \R$ such that for every geodesic $\gamma\in\gL$, 
$$U(\gamma) = \int_\gamma W(\xi) d\gamma(\xi),$$
where $d\gamma$ is the induced Riemannian measure on $\gamma\subset \Omega$.
\end{defn}

\subsection{Proof of Theorem \ref{theo:WZeta}}
By Theorem \ref{th:ConfCvxCpct}, $(M,g)$ is convex co-compact with pinched negative curvature. 
By the results of the previous section, the geodesic flow on $S^gM$ is Axiom A and its non-wandering set $\Omega$ is a hyperbolic basic set. Moreover, since the length spectrum is non-arithmetic, it follows from Theorem \ref{theo:Dalbo} that the flow is topologically weak-mixing on $\Omega(g)$. The extension property for the zeta function then follows from Theorem 11 of Haydn \cite{Hay90}. It follows from Section 6 of \cite{Pol86} that the extension to the half plane $\mathfrak{R}(s)> \wp(W)-\frac{\lambda\alpha}{2}$ is in general optimal (this was observed by Haydn).
\qed

When the weight function is trivial, we recover the \em Selberg zeta function \em 
$$\label{dz} Z(s) = \exp \left( \sum_{\gamma \in \pgL} \sum_{k \in \N} \frac{e^{-k s l_p(\gamma)}}{k} \right),$$
which was used to prove the eponymous trace formula for weakly symmetric spaces \cite{sel}. As long as it converges, this dynamical Zeta can also be represented as the following Euler product:
$$Z(s) = \prod_{\gamma \in \pgL} \left( 1 - e^{-sl_p (\gamma)} \right)^{-1}.$$

\begin{cor}[Extension of the dynamical zeta function]\label{cor:DZeta}
Let $(M, g)$ be a conformally compact manifold with negative sectional curvatures and with non Abelian fundamental group and non-arithmetic length spectrum. Then there exist positive constants $a$ and $b$ such that the sectional curvatures satisfy
$$-b^2 \leq K_g \leq -a^2 < 0,$$ 
and the dynamical Zeta function 
$$ Z(s) = \exp \left( \sum_{\gamma \in \pgL} \sum_{k \in \N} \frac{e^{-k s l_p(\gamma)}}{k} \right)$$ 
converges absolutely for $\mathfrak{R}(s) > h$.  It admits a meromorphic extension to $\mathfrak{R}(s) > h-\frac{\lambda}{2}$, where $\lambda\in[a,b]$ is the expansion factor of the geodesic flow on the non-wandering set $\Omega(g)$.  Moreover, with the exception of a simple pole at $h$, this extension is analytic and non-vanishing in an open neighborhood of $\{ \mathfrak{R}(s) \geq h \}$, where $h$ is the topological entropy of the geodesic flow. 
\end{cor}

\begin{remark}\label{rk:Opti1}
Is this extension for \emph{un-weighted} dynamical zeta functions optimal? It is not clear whether the examples constructed by Pollicott in \cite{Pol86} can be adapted to get convex co-compact manifolds whose standard dynamical zeta function has an essential singularity. Nevertheless, C. Guillarmou showed in \cite{Gui05} that for generic asymptotically hyperbolic $n$-manifolds, the resolvent of the Laplacian has an essential singularity in $(n/2 - \mathbb N)$. This strongly suggests that the meromorphic extension to $\mathfrak{R}(s) > h-\frac{\lambda}{2}$ should be optimal for general conformally compact manifolds.
\end{remark}

\begin{remark}
 When the flow is \emph{not topologically weak mixing}, it is possible to establish a meromorphic extension of the dynamical zeta function: this follows from the work of Parry and Pollicott \cite{pp83}. It allows the authors to prove a suitably written Prime Orbit Theorem, cf Theorem 2 of \cite{pp83}. We will not deal with the (hypothetical) case of non-mixing geodesic flows in this paper since it is expected that the geodesic flow on any convex co-compact manifold with pinched negative curvature is mixing. 
 \end{remark}

\subsection{Proof of Theorem \ref{theo:WPOT}} 
This theorem will be a direct consequence of Theorem \ref{theo:WZeta} and the following lemma, which is based on the Wiener-Ikehara \cite{w67} proof of the Prime Number Theorem and arguments of \cite{pp83}.

\begin{lemma}
Let $(M,g)$ be a convex co-compact manifold whose geodesic flow is topologically mixing, and $\gL$ be the set of its closed geodesics. Let $w : \gL\rightarrow [0,\infty)$ be a positive map such that the dynamical Zeta function
$$Z(s) = \exp \sum_{\gamma \in \cL_p} \sum_{\N} \frac{e^{-ksl(\gamma)} w(\gamma)^k}{k}$$
converges absolutely for $\mathfrak{R}(s) > \eta$ for some $\eta > 0$ and admits a non-vanishing analytic extension to an open neighborhood of $\mathfrak{R}(s) = \eta$ with the exception of a simple pole at $s = \eta$. Then
$$ \sum_{\gamma \in {\cL}_{T}} w(\gamma)\sim \frac{e^{\eta T}} {\eta T} \mbox{ when }T\rightarrow \infty.$$           
\end{lemma} 

\begin{proof} 
Let 
$$N(\gamma) = e^{\eta l(\gamma)}, \quad \Lambda(\gamma) = \log N(\gamma),$$
and 
$$\zeta(s) := \exp \left( \sum_{\cL_p} \sum_{\N} \frac{1}{k} N(\gamma)^{-ks} w(\gamma)^k \right).$$
By assumption, $\zeta$ has a simple pole at $s=1$ and admits a non-vanishing analytic extension (with the exception of the pole at $1$) to an open neighborhood of $\mathfrak{R}(s) =1$. Repeating the arguments in P. 588--589 of \cite{pp83}, it follows that when the flow is topologically mixing, 
$$\sum_{N(\gamma) \leq x} w(\gamma) \sim \frac{x}{\log x}.$$
\end{proof}

\begin{remark}
We note that the counting estimate for the length spectrum $\cL$ is equivalent to that for $\cL_p$. This follows from the calculation
$$\frac{e^{hx}}{hx} + o\left(\frac{e^{hx}}{hx}\right)\leq \# \{ \gamma \in \cL | l(\gamma) \leq x \} \leq \sum_{k=1} ^{[x]} \# \left\{ \gamma \in \cL_p | l(\gamma) \leq \frac{x}{k} \right\} \leq \frac{e^{hx}}{hx} + \frac{x e^{hx/2}}{hx/2}+ o\left(\frac{e^{hx/2}}{hx/2}\right).$$

An analogous calculation gives the equivalence of the weighted prime orbit theorem for $\cL$ and $\cL_p$. 
\end{remark}

An immediate corollary of the weighted prime orbit theorem (which follows by taking the trivial weight function $1$) gives the asymptotic growth of the number of closed geodesics. 
\begin{theorem}[Prime Orbit Theorem]\label{th:pot} 
Let $(M, g)$ be a conformally compact manifold with negative sectional curvature, with non Abelian fundamental group and non-arithmetic length spectrum. Then the length spectrum counting function satisfies
$$ \# \{ \gamma \in \gL : l(\gamma) \leq T \} \sim \frac{e^{hT}}{hT} \mbox{ when }T\rightarrow \infty.$$
\end{theorem} 

\begin{remark} The Prime Orbit Theorem also follows from the work of Roblin \cite{Rob03}, which applies in a much more general setting.
\end{remark} 

We now want to apply these results to the weighted Zeta function from \cite{gn} 
$$ \tilde{Z} (s) = \exp \left( \sum_{\gamma \in \pgL} \sum_{k \in \N} \frac{e^{-k s l_p(\gamma)}}{k \sqrt{| \det(I - \cP^k _{\gamma}) |}} \right),$$
where $\cP^k _{\gamma}$ is the $k$-times Poincar\'e map of the geodesic flow around the primitive closed orbit $\gamma$, described at the beginning of this section. However, the weight $$\gamma\mapsto \frac{1}{\sqrt{| \det(I - \cP^k _{\gamma}) |}}$$
does \emph{not} derive from a H\"older potential. To be able to deal with such weighted Zeta functions, we prove now that when two weights are asymptotically exponentially close, then their associated weighted Zeta functions have similar extension properties.

\begin{theorem} \label{theo:asyz}
Let $(M,g)$ be a convex co-compact manifold with non-arithmetic length spectrum, and $\gL$ be the set of its closed geodesics. Let $w,v : \gL \rightarrow [0,\infty)$ be two weights such that $w(k\gamma) = w(\gamma)^k$ and $v(k\gamma) = v(\gamma)^k$. We assume that there exist constants $C$, $\epsilon > 0$ with
$$v(\gamma) = w(\gamma)(1 + r(\gamma)), \quad |r(\gamma)| \leq C e^{-\epsilon k l(\gamma)}, \quad \forall \quad \gamma \in \cL_p \textrm{ and all } k \in \N.$$
Let 
$$Z(s) = \exp\left( \sum_{\gamma \in \cL_p} \sum_{k \in N} \frac{e^{-ksl(\gamma)} w(\gamma)^k}{k}\right) \mbox{ and } Z^*(s) = \exp \left( \sum_{\gamma \in \cL_p} \sum_{k \in \N} \frac{e^{-ksl(\gamma)} v(\gamma)^k}{k}\right)$$
be the weighted dynamical Zeta functions associated to $w$ and $v$. 
Assume that $Z$ converges absolutely for $\mathfrak{R}(s) > \eta$ for some $\eta \in \R$ and admits a non-vanishing analytic extension to an open neighborhood of $\mathfrak{R}(s) = \eta$ with the exception of a simple pole at $s = \eta$. 
\begin{enumerate}
\item $Z^*$ converges absolutely for $\mathfrak{R}(s) > \eta$ and admits a non-vanishing analytic extension to an open neighborhood of $\mathfrak{R}(s) = \eta$ with the exception of a simple pole at $s = \eta$.
\item If $\eta > 0$, then $w$ and $v$ satisfy the following counting estimates
$$ \sum_{\gamma \in \cL_{T}} w(\gamma)\sim \frac{e^{\eta T}} {\eta T} \mbox{ and } \sum_{\gamma \in \cL_{T}} v(\gamma)\sim \frac{e^{\eta T}} {\eta T} \mbox{ when }T\rightarrow \infty.$$ 
\end{enumerate}
\end{theorem}

\begin{proof}
By hypothesis, we may write 
$$Z^*(s) = \exp \sum_{\gamma \in \cL} \frac{e^{-sl(\gamma)} w(\gamma)}{k(\gamma)} (1 + r(\gamma)), \quad |r(\gamma)| \leq C e^{- \epsilon l(\gamma)},$$
where $k(\gamma)$ is defined to be $k$ if $\gamma = k \gamma_p$ for some $\gamma_p \in \cL_p$. So, we have
$$Z^*(s) = Z(s) . Z_{\epsilon} (s), \quad Z_{\epsilon}(s) = \exp \sum_{\gamma \in \cL} \frac{e^{-sl(\gamma)} w(\gamma) r(\gamma)}{k(\gamma)}.$$
By hypothesis, 
$$|Z_{\epsilon} (s)| \leq e^C Z(s+\epsilon).$$
By the absolute convergence of $Z$ for $\mathfrak{R}(s) > \eta$, it follows that $Z_{\epsilon}$ converges absolutely and is therefore analytic for $\mathfrak{R}(s) > \eta - \epsilon$. This implies (1). 
The statement (2) follows immediately from the preceding Lemma. 
\end{proof} 

\subsection{Proof of theorem \ref{theo:WZeta2}}
Let 
$$Z (s) := \exp \sum_{\gamma \in \cL} \frac{e^{-s l(\gamma)} w(\gamma)}{k(\gamma)},$$
where $\cL$ is the set of all closed geodesics and 
$$w(\gamma) = \exp\left(\int_{\gamma} \frac{-W_{SBR}}{2} \right).$$
Similarly, we write
$$\tilde{Z} (s) = \exp \sum_{\gamma \in \cL} \frac{e^{-s l(\gamma)} p(\gamma)}{k(\gamma)} , \mbox{ 
where }
p(\gamma) = |\det(I - \mathcal{P}_{\gamma})|^{-1/2}.$$
If the dimension of the manifold is $n+1$, and the sectional curvatures are bounded above by $-a^2$, then $\cP_{\gamma}$ has expanding eigenvalues $\lambda_1, \ldots, \lambda_n,$ and contracting eigenvalues $\lambda_{n+1}, \ldots, \lambda_{2n},$ and
$$|\det(I - \cP_\gamma)| = \prod_{1} ^{2n} |1 - \lambda_i| = \prod_{i=1} ^n |\lambda_i| \prod_{j=1} ^n \left|1 - \frac{1}{|\lambda_j|} \right| \prod_{k=n+1} ^{2n} |1 - \lambda_k|.$$
By Theorem 3.9.1 of \cite{Kli82}, 
$$|\lambda_i|^{-1} \leq e^{-a l(\gamma)}, \quad i=1, \ldots, n,$$
and 
$$|\lambda_i| \leq e^{-a l(\gamma)}, \quad i=n+1, \ldots 2n.$$
Therefore, there exists a constant $C>0$ such that 
$$|\det(I - \cP_\gamma)| = \prod_{i=1} ^n |\lambda_i|(1+R(\gamma)) \quad |R(\gamma)| \leq Ce^{-a l(\gamma)}, \quad \textrm{ for all } \gamma.$$
By definition of $W_{SBR}$, 
$$\exp\left(\int_{\gamma} W\right) = \prod_{i=1} ^n |\lambda_i|.$$
Thus, it also follows that 
$$p(\gamma) = w(\gamma)(1+r(\gamma)), \quad |r(\gamma)| \leq C e^{-a l(\gamma)} \textrm{ for all $\gamma$ with $l(\gamma) \gg 1$.}$$
Applying Theorem \ref{theo:WZeta} to $Z$ implies that it admits a non-vanishing analytic extension to an open neighborhood of 
$$\{ s \in \C : \mathfrak{R}(s) \geq \wp(-W_{SBR}/2)\},$$
with the exception of a simple pole at $\wp(-W_{SBR}/2)$. Moreover, 
\begin{equation} \label{eq:Za} \tilde{Z} (s) = Z(s) Z_a(s), \end{equation} 
where 
$$Z_a (s) = \sum_{\gamma \in \cL} \frac{e^{-s l(\gamma)} w(\gamma) r(\gamma)}{k(\gamma)}.$$
By the proof of Theorem \ref{theo:asyz} and the estimate on $r(\gamma)$, $Z_a$ converges absolutely and is therefore non-vanishing and analytic for $\mathfrak{R}(s) > \wp(-W_{SBR}/2) - a$. By Theorem \ref{theo:WZeta}, $Z$ admits a meromorphic extension to 
$$\left\{ \mathfrak{R}(s) > \wp(-W_{SBR}/2) - \frac{\lambda \alpha}{2} \right\},$$
where $\lambda$ is the expansion factor of the geodesic flow on the non-wandering set $\Omega$, and $\alpha$ is the H\"older exponent of $W$. By Proposition \ref{prop:WSBR}, 
$$\alpha = \inf \left\{ \frac{2a}{b}, 1\right\}.$$
Since $a \leq \lambda \leq b$, $\frac{\lambda \alpha}{2} \leq a$. It follows from (\ref{eq:Za}) and the absolute convergence of $Z_a$ for $\mathfrak{R}(s) > \wp(-W_{SBR}/2) - a$ that $\tilde{Z}$ extends meromorphically to 
$$\left\{ \mathfrak{R}(s) > \wp(-W_{SBR}/2) - \frac{\lambda \alpha}{2} \right\}.$$
In case $\wp(-\frac{W_{SBR}}{2} ) > 0$, the counting estimate follows from Theorem \ref{theo:asyz}. 
\qed

The arguments already pointed out in Remark \ref{rk:Opti1} strongly suggest that this extension may be optimal for general conformally compact manifolds, even though we could not provide a full proof of this.


\subsection{Applications}
We can apply the meromorphic extension of the dynamical zeta function to show that the entropy of the geodesic flow on a convex co-compact manifold changes analytically under an analytic perturbation of the original metric generalizing the work of A. Katok, G. Knieper, M. Pollicott and H. Weiss \cite{KKPW89}. Recall that for a real-analytic Riemannian manifold $(M,g)$, an \emph{analytic perturbation} of $g$ is a family $(g_\alpha)_{\alpha\in(-\epsilon,\epsilon)}$ for some $\epsilon>0$ such that the map $\alpha\mapsto g_\alpha$ is (real-)analytic. 

\begin{theorem}[Analyticity of the entropy]\label{theo:AnalEnt}
Let $(M,g)$ be a real-analytic convex co-compact manifold with pinched negative curvature, and $(g_\alpha)_{\alpha\in(-\epsilon,\epsilon)}$ be an analytic perturbation of $g$ by metrics of negative sectional curvatures. Then the topological entropy $h(g_\alpha)$ of the geodesic flow on $(M,g_\alpha)$ is an analytic function of $\alpha$.
\end{theorem}
\begin{proof}
Let $(M,g)$ be a real-analytic convex co-compact manifold with negative curvatures, and $(g_\alpha)_{\alpha\in(-\epsilon,\epsilon)}$ be an analytic perturbation of $g$. As long as the sectional curvatures of $g_\alpha$ remain strictly negative, the manifold $(M,g_\alpha)$ remains convex co-compact: this follows from Theorem 1.7 p.401 of \cite{BriHae99}. 
 This implies that there exists a compact set $K\subset M$ such that all closed geodesics of $(M,g_\alpha)_{-\epsilon<\alpha<\epsilon}$ are contained in $K$. Moreover, it follows from Section 2 of \cite{Bow73} that the dynamics of the geodesic flow on the non-wandering set of a compact manifold can be encoded by the symbolic dynamics of a \emph{shift over a finite set}. By Corollary \ref{cor:DZeta}, the dynamical Zeta 
 functions of the $(M,g_\alpha)$ extend meromorphically to the half-plane $\Re s>h(g_\alpha)-K$, where $K$ depends smoothly on the metric, with a simple pole in $h(g_\alpha)$. 

Therefore, the proof of Theorem 1 of \cite{KKPW89} can be reproduced {\it verbatim} and hence the map $\alpha\mapsto h(g_\alpha)$ is analytic.
\end{proof}

A convex co-compact hyperbolic 3-manifold $M = \mathbb H^3/\Gamma$ admits a family of analytic deformations, isomorphic to the Teichm\"uller space of its visual boundary (see \cite{Mar07} p.243 and references given there). Theorem \ref{theo:AnalEnt} above can be used to show that along any analytic path of hyperbolic structures on $M$, the entropy is analytic. 

\begin{cor}
Let $M = \mathbb H^3/\Gamma$ be a convex co-compact hyperbolic 3-manifold. Let $(\rho_\alpha)_{\alpha\in(-\epsilon,\epsilon)}$ be an analytic family of convex co-compact faithful discrete representations of $\Gamma$ into $PSL_2(\mathbb C)$, with $\rho_0=id$. Let $M_{\rho_\alpha(\Gamma)} = \mathbb H^3/\rho_\alpha(\Gamma)$ and $g_\alpha$ be the hyperbolic metric on each of these manifolds induced by the covering.
Then the topological entropy $\alpha\mapsto h(g_\alpha)$
is an analytic function.
\end{cor}

\begin{proof}
Let $M = \mathbb H^3/\Gamma$ be a convex co-compact hyperbolic 3-manifold, and let $(\rho_\alpha)_{\alpha\in(-\epsilon,\epsilon)}$,
$M_{\rho_\alpha(\Gamma)}$ and $g_\alpha$ be as in the statement above. By Theorem 4.2 of \cite{Tap10}, there exists an analytic family of metrics $\tilde{g}_\alpha$ on $M$ such that for all $\alpha\in(-\epsilon,\epsilon)$, the Riemannian manifolds $(M,\tilde{g}_\alpha)$ and $(M_{\rho_\alpha(\Gamma)}, g_\alpha)$ are isometric. By Theorem \ref{theo:AnalEnt}, the entropy $\alpha\mapsto h(\tilde{g}_\alpha)$ is analytic, which concludes the proof of this corollary.
\end{proof}

\begin{remark}
G. Contreras has shown in \cite{c92} that for compact manifolds with a hyperbolic flow, the pressure function and metric entropy both are $C^r$ in a $C^r$ neighborhood of the flow. His proof extends to convex co-compact manifolds, but does not provide the analyticity of the entropy. The analyticity of the entropy for convex co-compact hyperbolic 3-manifolds had been previously obtained for many special cases in \cite{AndRoch97}.
\end{remark}

\section{Interactions between dynamics and the Laplace spectrum}
Recall the Laplace operator $\Delta$ on an $n+1$ dimensional Riemannian manifold $(M,g)$; with respect to local coordinates $(x_1, \ldots, x_{n+1})$ 
$$\Delta = - \sum_{i,j=1} ^n \sqrt{\det(g)} \frac{\pa}{\pa x_i} g^{ij} \sqrt{ \det(g)} \frac{\pa}{\pa x_j}.$$
Given a complete manifold $(M,g)$, there is a canonical, unique self-adjoint operator (also denoted $\Delta$) on $\cL^2 (M)$ extending the Laplacian on smooth functions with compact support \cite{sull}. 

The spectral theory of conformally compact manifolds was inspired by Lax-Phillips \cite{lap} who studied the Laplacian on convex co-compact hyperbolic manifolds. This work was fundamental to Mazzeo in \cite{m88} and \cite{m91} who proved the following important result for the spectral theory of conformally compact manifolds which we recall below. 
\begin{theorem}[Mazzeo] \label{th:Mazzeo}
Let $(M,g)$ be an $n+1$ dimensional conformally compact manifold. Define 
$\alpha_0^2 := \inf\{ \lim_{p \to \infty} \kappa(p) \},$
where the infimum is taken over all sectional curvatures. Then, the essential spectrum of the Laplacian is absolutely continuous and is $\left[ \frac{\alpha_0^2 n^2}{4}, \infty \right).$
There are no embedded eigenvalues except possibly at $\alpha_0 ^2 n^2 /4$. 
\end{theorem} 

A connection between pure point spectrum and entropy of the geodesic flow was established by D. Sullivan in \cite{Sul79} who showed that, for hyperbolic $(n+1)$-manifolds, the pure-point spectrum $\sigma_{pp}(\Delta)$ is non-empty if and only if the entropy of the geodesic flow $h(g)$ satisfies $h(g)>n/2$. This result was further improved by Guillarmou-Naud, who demonstrated the following \cite{gn}. 

\begin{theorem}[Guillarmou-Naud, 2006]
Let $M = \bbH^{n+1}/\Gamma$ be a convex co-compact hyperbolic manifold whose entropy satisfies $h>n/2$. Then
$$\#\left\{\gamma\in\pgL : l(\gamma)\leq T\right\} = \li(e^{h T}) + \sum_{\beta_{n}(h)<\alpha_i<h}\li(e^{\alpha_i t}) + \calO\left(\frac{e^{\beta_n(h)T}}{T}\right),$$
where $\li(x) = \int_2^xdt/\log(t)$ and $\beta_n(h) = \frac{n}{n+1}(\frac{1}{2}+h)$. The coefficients $\alpha_i$ are in bijection with $\sigma_{pp}(\Delta)$ by $\alpha_i(n-\alpha_i) = \lambda_i\in\sigma_{pp}(\Delta)$.
\end{theorem}

In the more general case we consider, convex co-compact manifolds with variable curvature, there are few known results for the spectral theory. Our final result is inspired by \cite{Sul79} and the above result of Mazzeo. 

\subsection*{Proof of Theorem \ref{th:Eigenvalue}} 
Let $(M, g)$ be a convex co-compact manifold of dimension $n+1$ whose sectional curvatures $\kappa$ satisfy $-b^2 \leq K_g \leq -a^2<0$. Let $\Gamma$ be the fundamental group of $M$ acting on the universal cover $(\tilde M,g)$ by isometries, and $\Lambda_\Gamma\subset \pa_v \tilde M$ its limit set. Since the fundamental group of $M$ is not Abelian, the Patterson-Sullivan construction (cf. \cite{Ot-Pei}, p. 20) shows that there exists a family of finite positive measures $(\sigma_x)_{x\in \tilde M}$, supported by $\Lambda_\Gamma$, satisfying the following properties:
\begin{enumerate}
\item $(\sigma_x)$ is $\Gamma$-equivariant: for all $\gamma\in \Gamma$ and all $x\in \tilde M$, $\sigma_{\gamma^{-1} x} = \gamma^*\sigma_x$;
\item for all $x,y\in \tilde M$, the measures $\sigma_x$ and $\sigma_y$ are absolutely continuous with respect to each other, and satisfy for all $\xi \in \Lambda_\Gamma$
$$\frac{d\sigma_x}{d\sigma_y}(\xi) = e^{-h_g \mathcal B_\xi(x,y)},$$
where $\mathcal B_\xi(.,.)$ is the \emph{Busemann function} in $\xi$ and $h_g$ the topological entropy of the geodesic flow of $(M,g)$ restricted to its non-wandering set.
\end{enumerate}
Let $o\in \tilde M$ be fixed; we define the map $\tilde \phi : \tilde M\rightarrow(0,\infty)$ by
\begin{equation}\label{eq:phiPS}
\tilde \phi(x) = \int_{\Lambda_\Gamma} d\sigma_x(\xi) = \int_{\Lambda_\Gamma} e^{-h_g \mathcal B_\xi(x,0)}d\sigma_0(\xi).
\end{equation}

It follows from the $\Gamma$-equivariance of $(\sigma_x)$ that $\tilde \phi$ is $\Gamma$ equivariant. Therefore, it induced a well-defined positive map $\phi : M = \tilde M/\Gamma\rightarrow (0,\infty)$. We could not find a complete proof of the following classical lemma.
\begin{lemma}
If the sectional curvatures of $g$ satisfy $-b^2\leq K_g\leq -a^2<0$, then for all $\xi\in \pa_v \tilde M$ and $o\in \tilde M$, the Busemann function $x\mapsto \mathcal B_\xi(x,o)$ satisfies
$$na \leq \Delta_g \mathcal B_\xi(.,x)\leq nb.$$
\end{lemma}
\begin{proof}
Let $\xi\in \pa_v \tilde M$ and $o\in \tilde M$ be fixed. For all $x\in \tilde M$, we denote by $(\gamma_{x,\xi}(t))_{t\in(0,\infty)}$ the geodesic ray (parameterized with unit speed) which starts in $x$ and ends in $\xi$. For all $v\in T_xM$, we write $Y^s_v(t)$ the \emph{stable} Jacobi vector field along $\gamma_{x,\xi}$ with $Y^s_v(0) = v$. A detailed definition of stable Jacobi field can be found in \cite{HeiImHof}, p. 482. It follows from Proposition 3.1 of \cite{HeiImHof} that 
$$\nabla \mathcal B_\xi(.,o)(x) = - \gamma_{x,\xi}'(0) \mbox{ and } \nabla_v \nabla \mathcal B_\xi(.,o)(x) = -(Y^s_v)'(0)$$
for all $v\in T_xM$. Since 
$$\Delta_g\mathcal B_\xi(.,o)(x) = -\mbox{Trace}(v\mapsto \nabla_v \nabla \mathcal B_\xi(.,o))(x),$$
the Rauch Comparison Theorem (cf \cite{Kli82} p. 216) implies:
$na \leq \Delta_g \mathcal B_\xi(.,x)\leq nb.$ \end{proof}
By a straightforward computation, this lemma implies that for all $x\in \tilde M$ and $\xi\in \pa_v \tilde M$, we have
$$h_g(na-h_g) e^{-h_g \mathcal B_{\xi}(x,o)}\leq \Delta_g(e^{-h_g \mathcal B_{\xi}(.,o)})(x)\leq h_g(nb-h_g)e^{-h_g \mathcal B_{\xi}(x,o)}.$$
This implies in particular that on $M$,
$$\Delta \phi(x)\geq h_g(na-h_g) \phi(x).$$ 
Since $\phi$ is positive, Theorem 2.1 of \cite{sull} implies that the bottom of the spectrum of $\Delta_g$ satisfies $\lambda_0(\Delta_g)\geq h_g(na-h_g).$ 

When $h_g\leq \frac{na}{2}$, this lower bound can be improved as follows. For any $\delta\geq h_g$, it is shown in \cite{Ro11} p. 100 that there exists a family of finite positive measures $(\mu^\delta_x)_{x\in \tilde M}$, supported by $\tilde M\cup \pa_v \tilde M$, satisfying the following properties:
\begin{enumerate}
\item $(\mu^\delta_x)$ is $\Gamma$-equivariant: for all $\gamma\in \Gamma$ and all $x\in \tilde M$, $\mu^\delta_{\gamma^{-1} x} = \gamma^*\mu^\delta_x$;
\item for all $x,y\in \tilde M$, the measures $\mu^\delta_x$ and $\mu^\delta_y$ are absolutely continuous with respect to each other, and satisfy for all $\xi\in\tilde M\cup \pa_v \tilde M$
$$\frac{d\mu^\delta_x}{d\mu^\delta_y}(\xi) = e^{-\delta \mathcal B_\xi(x,y)},$$
where $\mathcal B_\xi(x,y) = d_g(x,\xi) - d_g(y,\xi)$if $\xi\in \tilde M$, and $\mathcal B_\xi(.,.)$ is the Busemann function in $\xi$ when $\xi\in \pa_v \tilde M$. 
\end{enumerate}

Moreover, it also follows from the proof of Proposition 3.1 of \cite{HeiImHof} that for all $\xi\in\tilde M\cup \pa_v \tilde M$, we have 
$$\Delta_g \mathcal B_\xi\geq na\mathcal B_\xi.$$
Therefore, the map $\phi^\delta : \tilde M\rightarrow (0,\infty)$ defined by
$$\phi^\delta(x) = \int_{\tilde M\cup \pa_v \tilde M}d\mu^\delta_x(\xi)$$
is positive, $G$-equivariant and satisfies 
$$\Delta_g \phi^\delta\geq \delta(na-\delta) \phi^\delta.$$
Since the map $\delta\mapsto \delta(na-\delta)$ is increasing on $[0,\frac{na}{2}]$, Theorem 2.1 of \cite{sull} implies that we have
$$\lambda_0(\Delta_g)\geq \frac{(na)^2}{4}.$$

The upper bound for $\lambda_0(\Delta_g)$ purely comes from the essential spectrum. It follows from \cite{Eich} that when $(M,g)$ is a convex co-compact $(n+1)$-manifold with infinite volume, with $-b^2\leq -a^2<0$, then the essential spectrum of $M$ contains the half line $(-\frac{(nb)^2}{4}, \infty)$. \qed

Since the upper bound on the bottom of the spectrum only depends on the essential spectrum, it can be improved to $\lambda_0(\Delta_g)\leq \frac{(n\beta)^2}{4}$ if the sectional curvatures satisfy $K_g\geq -\beta^2>-b^2$ on the complement of a compact set. When $M$ is conformally compact, Theorem \ref{th:Mazzeo} of Mazzeo improves this bound.


\subsection{Concluding Remarks}

Let us first remark that Theorem \ref{th:Eigenvalue} is by no mean optimal. In particular, it does not give a dynamical criterion for the existence of an isolated eigenvalue. Such a criterion may come from a further application of the study of the weighted zeta function given in Theorem \ref{theo:WZeta2}

\medskip

The counting estimates for the number of closed geodesics given in the prime orbit theorems can presumably be refined in a similar way to Theorem 1 of \cite{PoSh98}, where M. Pollicott and R. Sharp adapt the work of D. Dolgopiat to negatively curved \emph{compact surfaces}. The results announced in \cite{st} extend this work of Pollicott and Sharp to Axiom A flows which satisfy some non-integrability conditions. These conditions are satisfied by the geodesic flow on complete convex co-compact surfaces and compact manifolds with pinched negative curvature. However, it is still unknown whether the geodesic flow on \emph{convex co-compact manifolds of dimension at least 3} satisfies Stoyanov's non-integrability condition. Stoyanov's work gives an asymptotic expansion for the number of closed geodesics as a sum of exponential terms, which shall be compared to the result of Guillarmou-Naud quoted in our introduction. If the counting estimates for weighted geodesics could be refined, it would presumably lead to conditions for the existence of \emph{several distinct eigenvalues} in the pure point spectrum of the Laplacian.


\begin{thebibliography}{99}
\bibitem[AR97]{AndRoch97} J. Anderson and A. Rocha,\emph{ Analyticity of {H}ausdorff dimension of limit sets of {K}leinian group,} Ann. Acad. Sci. Fenn. Math. 22:2, (1997), 349--364.

\bibitem[A05]{and2} M. T. Anderson, {\em Geometric aspects of the AdS/CFT correspondence,} AdS/CFT correspondence: Einstein metrics and their conformal boundaries, 1--31, IRMA Lect. Math. Theor. Phys. 8, Eur. Math. Soc. Z\"urich, (2005). 

\bibitem[Be87]{Bes} A. Besse, \em Einstein Manifolds, \em Springer Verlag (1987). 

\bibitem[Ba09]{ba09} E. Bahuaud, \em Intrinsic characterization for Lipschitz asymptotically hyperbolic metrics, \em Pacific J. Math. 239, no. 2, (2009), 231--249.




 

\bibitem[B-P09]{bp} D. Borthwick and P. Perry, \em Inverse scattering results for manifolds hyperbolic near infinity, \em arXiv:0906.0542v2, (2009). 

\bibitem[Bowd95]{Bow95}
B.~H. Bowditch,\em Geometrical finiteness with variable negative curvature, \em Duke Math. J., 77:1, (1995), 229--274.



\bibitem[Bow73]{Bow73} R. Bowen, \em Symbolic dynamic for hyperbolic flows, \em Am. J. Math., 95 (1973), 429--459.


\bibitem[Bow75]{Bow75} R. Bowen, \em Equilibrium states and the Ergodic Theory of Anosov Diffeomorphisms, \em Lect. Notes Math. 470 (1975), Springer-Verlag. 

\bibitem[B-R75]{BoRu75} R. Bowen and D. Ruelle, \em The ergodic theory of Axiom A flows, \em Invent. Math., 29 (1975), 181-202.

\bibitem[Br-H99]{BriHae99} M. Bridson and A. Haefliger, \emph{Metric spaces of non-positive curvature}, Grund. Math. Wiss., Springer-Verlag, 319, (1999).


\bibitem[C92]{c92} G. Contreras, {\em Regularity of topological and metric entropy of hyperbolic flows,} 
Math. Z. 210, no. 1, (1992), 97--111. 

\bibitem[Dal00]{Dal} F. Dalb'o, {\em Topologie du feuilletage stable,} 
Ann. Inst. Fourier 50:3 (2000), 981--993. 

\bibitem[D-G75]{dg} J.J Duistermaat and V.W. Guillemin, \em The spectrum of Positive Elliptic operators and periodic bicharacteristics, \em Inv. Math., 29 (1975), 39--79.

\bibitem[Eich84]{Eich} J. Eichhorn, {\em The essential spectrum of non-simply connected open complete Riemannian manifolds}, Ann. Glob. Anal. Geom. Vol. 2:1 (1984), 1--18

\bibitem[E72]{e72} P. Eberlein, {\em Geodesic flows on negatively curved manifolds I}, Ann. of Math., 2:95, (1972), 492--510.

\bibitem[E73]{e73} P. Eberlein, {\em Geodesic flows on negatively curved manifolds II}, Trans. Amer. Math. Soc., 178 (1973), 57--82.

\bibitem[F-G85]{fg} C. Fefferman and C. R. Graham, \em Conformal invariants. The mathematical heritage of \'Elie Cartan, \em (Lyon, 1984), Ast\'erisque Num\'ero Hors S\'erie, (1985), 95--116.


\bibitem[Gui05]{Gui05} C. Guillarmou, \emph{Meromorphic properties of the resolvent on asymptotically hyperbolic manifolds}, Duke Math. J., 129:1 (2005), 1--35

\bibitem[G-N06]{gn} C. Guillarmou and F. Naud, {\em Wave 0-Trace and length spectrum on convex co-compact hyperbolic manifolds,} Comm. Anal. Geometry, 14:5, (2006), 945--967.

\bibitem[G-Z99]{gz} L. Guillop\'e and M. Zworski, {\em The wave trace
for Riemann surfaces,} Geom. Funct. Anal., 9:6, (1999), 
1156--1168.


\bibitem[Ha94]{Has94} B. Hasselblatt, \emph{Horospheric foliations and relative pinching},
J. Diff. Geom., 39:1, (1994), 57--63. 

\bibitem[Hay90]{Hay90} N. Haydn, \emph{Meromorphic extension of the Zeta function for Axiom A flows}, Ergodic Theory Dynamical Systems, 10 (1990), 347--360.

\bibitem[H-IH77]{HeiImHof} E. Heintze and H.C. Im Hof, \emph{Geometry of Horospheres}, J. Diff. Geom., 12 (1977), 481--491

\bibitem[H-P75]{HirPug75} M. Hirsch and C. Pugh, \emph{Smoothness of horocycle foliations},
J. Diff. Geom., 10 (1975), 225--238. 



\bibitem[K-K-P-W89]{KKPW89}
A. Katok, G. Knieper, M. Pollicott, and H.~Weiss, \emph{Differentiability and analyticity of topological entropy for {A}nosov and geodesic flows, }
Invent. Math., 98:3, (1989), 581--597.

\bibitem [Kl82]{Kli82}
W. Klingenberg, \emph{Riemannian geometry}, de Gruyter Studies in Mathematics, (1982). 

\bibitem[L-P-82]{lap} P. Lax and R. Phillips, \emph{The Asymptotic Distribution of Lattice Points in Euclidean and Non-Euclidean Spaces}, J. Func. Ana. 46 (1982), 280--350.

\bibitem[Mar07]{Mar07}
A. Marden,{\em Outer circles, An introduction to hyperbolic 3-manifolds,} Cambridge University Press, Cambridge, (2007).


\bibitem[M88]{m88} R. Mazzeo, \em The Hodge cohomology of a conformally compact metric, \em J. Diff. Geom, 28, (1988), 309--339. 

\bibitem[M91]{m91} R. Mazzeo, \em Unique continuation at infinity and embedded eigenvalues for asymptotically hyperbolic manifolds, \em Amer. J. Math, 113, (1991), 25--45. 

\bibitem[M-M87]{mmah} R. Mazzeo and R. B. Melrose, {\em Meromorphic extension of the resolvent on complete spaces with asymptotically constant negative curvature,} J. Funct. Anal., 75, (1987), 260--310.


\bibitem[O-P04]{Ot-Pei} J.P. Otal and M. Peign\'e, \em Principes variationnels et Groupes Kleiniens, \em Duke Math. J., 125:1 (2004), 15--44


\bibitem[P-P83]{pp83} W. Parry and M. Pollicott, {\em An analogue of the prime number theorem for closed orbits of Axiom A flows,}
Ann. of Math.,2, 118:3 (1983), 573--591.

\bibitem[P-P-Ast]{Par-Pol-Aster} W. Parry and M. Pollicott, {\em Zeta functions and the periodic orbit structure of hyperbolic dynamics.} 
Ast\'erisque No. 187-188 (1990), 268 pp.

\bibitem[P-P-E01]{pp} S.J. Patterson and P. Perry with an appendix by C. Epstein, {\em Divisor of the Selberg Zeta function for Kleinian groups in even dimensions,} Duke Math. J., 326, (2001), 321--390.


\bibitem[Pe01]{plength} P. Perry, \em Asymptotics of the length spectrum for hyperbolic manifolds of infinite volume, \em Geom. Funct. Anal., 11:1, (2001), 132--141. 

\bibitem[Po86]{Pol86} M. Pollicott, {\em Meromorphic extensions of generalized Zeta functions,} Inv.Math., 85, (1986), 147--164.

\bibitem[PoSh98]{PoSh98} M. Pollicott and R. Sharp, \emph{Exponential error terms for growth functions on negatively curved surfaces}, Amer. J. Math. 120 (1998), no. 5, 1019--1042.

\bibitem[Rob03]{Rob03}
T. Roblin, \emph{Ergodicit\'e et \'equidistribution en courbure n\'egative. (French) [Ergodicity and uniform distribution in negative curvature]} M\'em. Soc. Math. Fr. (N.S.) No. 95 (2003). 

\bibitem[Rob11]{Ro11}
T. Roblin, \emph{Comportement harmonique des densit\'es conformes et fronti\'ere de Martin} Bull. Soc. Math. Fr. 139:1 (2011), 97--128.

\bibitem[Ro09]{row1} J. Rowlett, \em Dynamics of asymptotically hyperbolic manifolds, \em Pacific J. Math., 242:2, (2009), 377--397. 



\bibitem[Se56]{sel} A. Selberg, \em Harmonic analysis and discontinuous groups in weakly symmetric Riemannian spaces with applications to Dirichlet series, \em J. Indian Math. Soc. (N.S.), 20, (1956), 47--87.

\bibitem[Sm67]{sm} S. Smale, \em Differentiable dynamical systems, \em Bull. Amer. Math. Soc., 73, (1967), 747--817.

\bibitem[St10]{st} L. Stoyanov, \em Spectra of Ruelle transfer operators for Axiom A flows on basic sets, \em arXiv:0810.1126v4 (2010). 

\bibitem[Su79]{Sul79}D. Sullivan, \emph{The density at infinity of a discrete group of hyperbolic motions,} Inst. Hautes \'Etudes Sci. Publ. Math., 50, (1979), 171--202.

\bibitem[Su87]{sull} D. Sullivan, \em Related aspects of positivity in Riemannian geometry, \em J. Differential Geometry, 25 (1987), 327--351. 

\bibitem[T10]{Tap10}
S. Tapie, \emph{A variation formula for the topological entropy of convex-cocompact manifolds}, to appear in Ergodic Theory Dynamical Systems.

\bibitem[W67]{w67} N. Wiener, \em The Fourier Integral and Certain of its Applications, \em C. U. P., Cambridge, England, (1967). 

\bibitem[Yau73]{Yau73}
S. T. Yau, \emph{Remarks on conformal transformations,} 
J. Diff. Geom., 8, (1973), 369--381.



\end{thebibliography}
\end{document}